\documentclass{amsart}

\usepackage{amsmath, amssymb, amsthm, mathtools, tikz-cd, enumerate}
\usepackage[hidelinks]{hyperref}

\theoremstyle{plain}
\newtheorem{thm}{Theorem}
\newtheorem{prop}[thm]{Proposition}
\newtheorem{lem}[thm]{Lemma}

\theoremstyle{definition}
\newtheorem{defi}[thm]{Definition}
\newtheorem*{acknowledgments}{Acknowledgments}

\newcommand{\cH}{\mathcal{H}}
\newcommand{\cF}{\mathcal{F}}
\newcommand{\cK}{\mathcal{K}}
\newcommand{\cO}{\mathcal{O}}

\newcommand{\clo}[1]{\overline{#1}}

\newcommand{\into}{\hookrightarrow}

\newcommand{\restrict}[2]{\left.{#1}\right|_{#2}}
\newcommand{\norm}[1]{\left\lVert {#1}\right\rVert}
\newcommand{\abs}[1]{\left| {#1}\right|}
\newcommand{\ndim}[1]{\operatorname{dim}_\mathrm{nuc}({#1})}

\title{Nuclear dimension of graph $C^*$-algebras with Condition~(K)}
\author{Gregory Faurot and Christopher Schafhauser}
\address{Department of  Mathematics, University of Nebraska - Lincoln, Lincoln, NE, USA}
\email{gfaurot2@huskers.unl.edu}
\email{cschafhauser2@unl.edu}
\thanks{This project was partially supported by the second author's NSF grant DMS-2000129.}

\date{\today}


\begin{document}
	
	\begin{abstract}
		We prove that for any countable directed graph $E$ with Condition~(K), the associated graph $C^*$-algebra $C^*(E)$ has nuclear dimension at most $2$. Furthermore, we provide a sufficient condition producing an upper bound of $1$.
	\end{abstract}
	
	\maketitle
	
	\renewcommand{\thethm}{\Alph{thm}}
	\section*{Introduction}
	The notion of nuclear dimension, introduced by Winter and Zacharias in \cite{WZ10}, is a generalization of topological covering dimension to $C^*$-algebras.  A $C^*$-algebra has nuclear dimension at most $n \geq 0$ if there are ``($n$+1)-colored'' finite rank approximations of the identity map, in the sense that there are a finite dimensional $C^*$-algebra $F$, a completely positive, contractive map $\psi \colon A \rightarrow F$, and a completely positive map $\phi \colon F \rightarrow A$ such that $\phi \circ \psi$ approximates the identity map pointwise in norm and $\phi$ decomposes as a sum of $n+1$ orthogonality preserving completely positive, contractive maps with mutually orthogonal supports (see Definition~\ref{dimnuc}).  
 
    Nuclear dimension has proved invaluable in the classification of simple, nuclear $C^*$-algebras; see \cite{GLN20A, GLN20B, EGLN23, TWW17}, for example.  In particular, the exotic examples of simple, nuclear $C^*$-algebras constructed in \cite{Villadsen1, Villadsen2, Rordam, Toms} have infinite nuclear dimension.  This led to the finiteness of the nuclear dimension becoming a key regularity hypothesis in Elliott's classification program.
    
    For separable, unital, simple, nuclear, infinite dimensional $C^*$-algebras, the Toms--Winter conjecture (\cite[Conjecture~9.3]{WZ10}; cf.\ \cite{Elliott-Toms, Winter10}) states that finite nuclear dimension is equivalent to tensorial absorption of the Jiang--Su algebra $\mathcal{Z}$, defined in \cite{JiangSu}.  The remarkable results of Castillejos, Evington, Tikuisis, White, and Winter in \cite{CETWW21, CE20}, building on a long line of work going back to Matui and Sato in \cite{MatuiSato1, MatuiSato2}, have completely determined the possible values of nuclear dimension for separable, simple $C^*$-algebras. The papers \cite{CETWW21, CE20} state that the nuclear dimension of a separable, simple $C^*$-algebra $A$ is given as follows:
	$$\ndim{A}=\begin{cases}
		0 & \text{ if $A$ is AF;} \\
		1 & \text{ if $A$ is nuclear, $\mathcal{Z}$-stable, and not AF;}\\
		\infty & \text{ otherwise.}
	\end{cases}$$
	
	The nuclear dimension of non-simple $C^*$-algebras is less understood.  For a compact metric space $X$, the nuclear dimension of $C(X)$ is precisely $\dim(X)$.  However, outside the type I setting, there is growing evidence that the nuclear dimension of a $C^*$-algebra is either infinite or very small.  For example, using Gabe's generalization of the Kirchberg--Phillips theorem \cite{Gabe1,Gabe2}, it was shown by Bosa, Gabe, Sims, and White in \cite{BGSW22} that a separable, nuclear, $\mathcal O_\infty$-stable $C^*$-algebra has nuclear dimension 1, improving Szab\'o's earlier upper bound of 3 obtained in \cite{Szabo}.  In the stably finite setting, Tikuisis and Winter showed $C(X) \otimes \mathcal Z$ has nuclear dimension 1 or 2 (\cite[Theorem~4.1]{TW14}).
    
    Finding the precise values of nuclear dimension has proven to be a difficult problem.  In Winter and Zacharias's original paper on nuclear dimension (\cite{WZ10}), they proved the nuclear dimension of the Toeplitz algebra $\mathcal{T}$ to be either $1$ or $2$.  The lower bound follows since $\mathcal T$ is not AF.  The upper bound arises from realizing $\mathcal T$ as an extension of $C(\mathbb T)$ by $\mathcal K$, the $C^*$-algebra of compact operators on a separable, infinite dimensional Hilbert space, and combining a 2-colored approximation of $C(\mathbb T)$ with a 1-colored approximation of $\mathcal K$ to obtain a 3-colored approximation of $\mathcal T$.  The nuclear dimension of $\mathcal T$ was eventually shown to be 1 by Brake and Winter in \cite{BW19} by introducing a method of reusing one of the colors in the approximation of $C(\mathbb T)$ to approximate $\mathcal K$.  
    The Brake--Winter technique was refined in \cite{GT22} to show that for a compact metric space $X$, a unital, essential extension of $C(X)$ by $\mathcal K$ has nuclear dimension $\dim(X)$.  
    Further, combining the Brake--Winter technique with Gabe's generalization of the Kirchberg--Phillips theorem, the Cuntz--Toeplitz algebras $\mathcal T_n$ were shown to have nuclear dimension 1 in \cite{15A20}.  This was further improved by Evington in \cite{E22}, producing an improved bound on nuclear dimension of certain extensions. In particular, any full extension of an $\cO_\infty$-stable $C^*$-algebra by a stable AF-algebra has nuclear dimension $1$.
    
	The present work is concerned with the problem of computing the precise value of nuclear dimension for graph $C^*$-algebras.  Given a directed graph $E$, Kumjian, Pask, and Raeburn constructed an associated $C^*$-algebra $C^*(E)$ in \cite{KPR98}, defined by associating each edge to a partial isometry with restrictions on the range and source projections determined by the structure of the graph (see Definition~\ref{ckdef}).  Graph $C^*$-algebras and related objects have become a natural test case for obtaining optimal bounds for nuclear dimension.  For example, the results of \cite{RSS15}, which builds on \cite{Enders}, proving all UCT Kirchberg algebras have nuclear dimension 1 via examining the nuclear dimension of certain graph $C^*$-algebras (or rather, 2-graph $C^*$-algebras) motivated the push to show that all Kirchberg algebras have nuclear dimension 1 in \cite{BBSTWW19} (improving the previous upper bound of 3 in \cite{MatuiSato2}).
 
    For a finite graph $E$, it follows from the permanence properties of finite nuclear dimension established by Winter and Zacharias in \cite{WZ10} and the general theory of graph $C^*$-algebras that $C^*(E)$ has finite nuclear dimension.  Indeed, the structure of gauge-invariant ideals in graph $C^*$-algebras (see \cite[Theorem~4.1]{BPRS00}) implies that all such graph $C^*$-algebras are defined by recursively taking extensions of AF-algebras, UCT Kirchberg algebras, and $C^*$-algebras stably isomorphic to $C(\mathbb T)$, all of which have finite nuclear dimension.  These naive upper bounds are typically not sharp; this is already seen in the cases of the Toeplitz and Cuntz--Toeplitz algebras discussed above.  In fact, we do not know of any graph $E$ where the nuclear dimension of $C^*(E)$ is known to be bigger than 1.

    We restrict our attention to graphs with Condition~(K), as introduced in \cite{KPR98, DT05}, which can be regarded as a kind of freeness condition on the dynamics of the graph---the definition is recalled in Definition~\ref{defK} below.  At the $C^*$-algebraic level, a graph $E$ has Condition~(K) if and only if all ideals of $C^*(E)$ are invariant under the gauge action, or equivalently, no subquotient of $C^*(E)$ is stably isomorphic to $C(\mathbb T)$.  In this case, Ruiz, Sims, and Tomforde showed in \cite[Theorem~5.1]{RST15} that, under the additional restriction that every vertex in $E$ receives a path from a  cycle in $E$, the nuclear dimension of $C^*(E)$ is at most 2.  The first of our main results removes this combinatorial restriction, obtaining the same bound for all graphs with Condition~(K).  This is the first finite bound which covers all graphs with Condition~(K).
    
	\begin{thm}\label{thmA}
		If $E$ is a countable directed graph with Condition~{\rm (K)}, then $C^*(E)$ has nuclear dimension at most $2$.
	\end{thm}

    Using an inductive limit argument motivated by \cite{RS04} and \cite{JP02} and the lower 
    semicontinuity of nuclear dimension under inductive limits, Theorem~\ref{thmA} reduces to the case of finite graphs.  In this setting, under a further combinatorial restriction on the graph, we obtain the following improvement.  For a vertex $v$ and a cycle $\mu$ in a graph $E$, we say $v$ \emph{connects to} $\mu$ if there is a path in $E$ with source $v$ whose range is a vertex on $\mu$. The extra conditions in Theorem~\ref{thmB} are similar in flavor to that of Ruiz, Sims, and Tomforde in \cite{RST15} but are unrelated and serve a different role in the proof.
	
    \begin{thm}\label{thmB}
		Let $E$ be a finite graph with Condition~{\rm (K)} such that each source $v$ in $E$ satisfies one of the following conditions:
		\begin{enumerate}
			\item\label{B1} $v$ connects to every cycle in $E$;
			\item\label{B2} $v$ connects to no cycles in $E$.
		\end{enumerate}
		Then $C^*(E)$ has nuclear dimension at most $1$.
	\end{thm}
 
	The bound in Theorem~\ref{thmB} is optimal. In the setting of this theorem, the nuclear dimension of $C^*(E)$ is characterized as follows: 
	$$\ndim{C^*(E)}=\begin{cases}
			0 & \text{ if $E$ has no cycles;} \\
			1 & \text{ if $E$ has a cycle.}
		\end{cases}$$
 
    Our techniques are combinatorial in nature, relying upon the structure of the graphs themselves.  In the case of a finite graph $E$ with Condition~(K), let $I$ denote the ideal of $C^*(E)$ generated by the projections corresponding to sources in the graph, and consider the extension
    \begin{equation*}
        0 \longrightarrow I \longrightarrow C^*(E) \longrightarrow C^*(E)/I \longrightarrow 0.
    \end{equation*}
    We show that $I$ is AF and $C^*(E) / I$ is $\mathcal O_\infty$-stable, which yields the upper bound of 2 in Theorem~\ref{thmA}.  Further, in the setting of Theorem~\ref{thmB}, after removing a finite dimensional direct summand of $C^*(E)$, we reduce to the case in which all sources in $E$ are as in Theorem~\ref{thmB}\ref{B1}. In this case, the extension above is full and $I$ is stable.  Then Evington's result from \cite{E22} lowers the bound on the nuclear dimension of $C^*(E)$ from 2 to 1, proving Theorem~\ref{thmB}. The inductive limit argument of Section~\ref{sec:reduction} involves adding sources to finite subgraphs of an infinite graph $E$. We do not know of a natural condition to place on $E$ so that these additional sources satisfy the hypotheses of Theorem~\ref{thmB}.

    After recalling some preliminary material on graph $C^*$-algebras and nuclear dimension in Section~\ref{sec:prelim}, we establish the reduction to finite graphs in Section~\ref{sec:reduction}.  Sections~\ref{sec:thmA} and~\ref{sec:thmB} are devoted to the proofs of Theorems~\ref{thmA} and~\ref{thmB}, respectively.

\setcounter{thm}{0}
\numberwithin{thm}{section}

\begin{acknowledgments}
The authors would like to thank Stuart White for an enlightening discussion on nuclear dimension and the contents of \cite{BGSW22} and \cite{15A20}. We also thank the referee for their helpful comments.
\end{acknowledgments}

\section{Preliminaries}\label{sec:prelim}
A completely positive, contractive (cpc) map $\phi\colon A \to B$ is \emph{order zero} if for all $a,b \in A$ with $ab=0$, we have $\phi(a)\phi(b)=0$ (see \cite{KW04}). Winter and Zacharias gave the following definition of nuclear dimension in \cite{WZ10}.

\begin{defi}\label{dimnuc}
Given a $C^*$-algebra $A$, the \emph{nuclear dimension} of $A$, denoted $\mathrm{dim}_\mathrm{nuc}(A)$, is at most $n$ if, given a finite set $\cF \subset A$ and $\epsilon>0$, there exist finite dimensional $C^*$-algebras $F_0, \dots, F_n$ and maps \begin{equation*} 
A \overset\psi\longrightarrow \bigoplus_{i=0}^n F_i \overset\phi\longrightarrow A,
\end{equation*}so that
\begin{enumerate}
	\item $\norm{a-\phi(\psi(a))}<\epsilon$ for all $a \in \cF$,
	\item $\psi$ is cpc, and
	\item $\restrict{\phi}{F_i}$ is cpc and order zero for all $i$.
\end{enumerate}
\end{defi}
It is easy to see that finite dimensional $C^*$-algebras have nuclear dimension 0. By \cite[Proposition~2.3]{WZ10}, it follows that all AF-algebras have nuclear dimension 0. The converse holds by \cite[Example~6.1(i)]{KW04}; that is, any separable $C^*$-algebra with nuclear dimension 0 is AF.

Throughout this paper, $E=(E^0, E^1, r,s)$ will be a directed graph with vertex set $E^0$, edge set $E^1$, and range and source maps $r,s\colon E^1 \to E^0$. A graph $E$ is \emph{row-finite} if $\abs{r^{-1}(v)}<\infty$ for all $v \in E^0$. The following definition is due to \cite{KPR98, DT05}, although we are using the convention of Raeburn's book \cite{R05}. 

\begin{defi}\label{ckdef}
	A \emph{Cuntz--Krieger $E$-family} is a family $(s,p)$ of mutually orthogonal projections $\{p_v: v\in E^0\}$ and partial isometries $\{s_e: e\in E^1\}$ with pairwise orthogonal ranges subject to the following conditions:
	\begin{enumerate}
		\item $s_e^*s_e=p_{s(e)}$; \label{ckdef1}
		\item $p_{r(e)}s_e=s_e$ for all $e \in E^1$; \label{ckdef2}
		\item $p_v =\sum_{\{e \in E^1: r(e)=v\}}s_es_e^*$ for every $v \in E^0$ such that $0<\abs{r^{-1}(v)}<\infty$. \label{ckdef3}
	\end{enumerate}
\end{defi}
The graph $C^*$-algebra $C^*(E)$ is the universal $C^*$-algebra generated by a universal Cuntz--Krieger $E$-family $(s,p)$. We refer the reader to \cite{R05} for additional background on graphs and their $C^*$-algebras. Next, we describe Condition~(K) of Kumjian, Pask, Raeburn, and Renault (\cite{KPRR97}) and their characterization of ideals of $C^*(E)$ when $E$ is row-finite and satisfies Condition~(K).  A \emph{return path} for a vertex $v$ is a path $\mu=\mu_1\mu_2\dots\mu_m$ with $s(\mu)=r(\mu)=v$ and $r(\mu_i)\neq v$ for $1 < i \leq m$.

\begin{defi}\label{defK}
A directed graph $E$ is said to have \emph{Condition}~(K) if  for each vertex $v \in E^0$, one of the following conditions holds:
\begin{enumerate}
	\item $v$ does not lie on a cycle in $E$;
	\item there are at least two return paths for $v$.
\end{enumerate}
\end{defi}

For a row-finite graph $E$ with Condition~(K), the ideals of $C^*(E)$ are in bijection with certain subsets of $E^0$, which we now describe.

\begin{defi}
	A subset $H \subset E^0$ is said to be \emph{hereditary} if whenever $v \in H$, $w \in E^0$, and there exists a path from $w$ to $v$, we have $w \in H$. A \emph{saturated} subset $H$ contains all vertices $v$ satisfying $r^{-1}(v) \neq \emptyset$ and $s(r^{-1}(v)) \subset H$. Given a hereditary, saturated subset $H \subset E^0$,  define the graphs $E_H\coloneqq(H,r^{-1}(H),r|_H, s|_H)$ and $E\setminus H\coloneqq(E^0\setminus H, s^{-1}(E^0\setminus H), r|_{E^0\setminus H}, s|_{E^0 \setminus H})$.  Note that the restrictions of the range and source maps are well-defined since $H$ is saturated and hereditary.
\end{defi}

The following reformulation of \cite[Theorem~6.6]{KPRR97}, found as \cite[Theorem~4.9]{R05}, gives a bijection between saturated, hereditary subsets of $E^0$ and ideals of $C^*(E)$ in the case when $E$ is a row-finite graph with Condition~(K).

\begin{thm}\label{K_ideals}
	Let $E$ be a row-finite graph satisfying Condition~{\rm (K)}. For each saturated, hereditary subset $H \subset E^0$, define $I_H$ to be the ideal of $C^*(E)$ generated by $\{p_v:v \in H\}$. Then $$I_H=\clo{\mathrm{span}}\{s_\mu s_\nu^*: s(\mu)=s(\nu) \in H\}.$$
    Further, $H \mapsto I_H$ is an isomorphism between the lattice of saturated, hereditary subsets of $E^0$ and the lattice of ideals of $C^*(E)$. The quotient $C^*(E)/I_H$ is naturally isomorphic to $C^*(E\setminus H)$, and the graph algebra $C^*(E_H)$ is isomorphic to a full corner of $I_H$.
\end{thm}

Finally, we will require some extension theory to prove Theorem~\ref{thmB}. For a thorough treatment of extensions, we refer the reader to \cite[Chapter~VII]{B98}. Given a $C^*$-algebra $A$, let $M(A)$ denote the multiplier algebra of $A$, and let $Q(A)\coloneqq M(A)/A$ be the corona algebra of $A$. For $C^*$-algebras $I$ and $B$, an \emph{extension of $B$ by $I$} is a $C^*$-algebra $A$ with a short exact sequence of the form
$$0 \longrightarrow I \longrightarrow A \overset\pi\longrightarrow B  \longrightarrow 0.$$
For any such extension, the inclusion $I \into M(I)$ can be canonically extended to a $^*$-homomorphism $\lambda \colon A \to M(I)$. This produces an associated \emph{Busby map} $\beta\colon B \to Q(I)$. An extension is \emph{full} if the Busby map of the extension is full in the sense that $\beta(b)$ generates $Q(I)$ as an ideal for all non-zero $b \in B$. If $a \in A\setminus I$ is such that $\lambda(a)$ is full in $M(I)$, then $\beta(\pi (a))$ is full in $Q(I)$ since $\beta \circ \pi = \pi_I \circ \lambda$, where $\pi_I\colon M(I)\to Q(I)$ is the quotient map. Therefore, to show an extension is full, it suffices to show that $\lambda$ is full.

\section{A Reduction to Finite Graphs}\label{sec:reduction}

In this section, we will prove that Theorem~\ref{thmA} can be reduced to the case of finite graphs. In the following section, it will be proven that $\ndim{C^*(E)}\leq 2$ whenever $E$ is a finite graph with Condition~(K). For countably infinite graphs with Condition~(K), an inductive limit approximation $\varinjlim C^*(E_i) = C^*(E)$, for a suitable sequence of finite subgraphs $E_i \subset E$, produces an upper bound on the nuclear dimension as 
$$\ndim{\varinjlim \, C^*(E_i)}\leq \liminf \ndim{C^*(E_i)} \leq 2$$
by \cite[Proposition~2.3]{WZ10}.  Throughout this section, $E$ will be a row-finite graph with Condition~(K); although we do not explicitly assume $E$ is infinite, this will be the case of interest.  

We first construct arbitrarily large finite subgraphs of $E$ with Condition~(K).

\begin{lem}\label{K_subgraph}
	Let $E$ be a row-finite directed graph with Condition~{\rm (K)} and let $F \subset E$ be a finite subgraph. Then there is a finite subgraph $F' \subset E$ containing $F$ which satisfies Condition~{\rm (K)} and the additional property that every vertex of $F$ lying on a cycle in $E$ has at least two distinct return paths in $F'$.
\end{lem}

\begin{proof}
	If a vertex $v \in F^0$ lies on a cycle in $E$, find distinct return paths $\mu^v$ and $\nu^v$, and let $F_v$ be the subgraph of $E$ consisting of the return paths $\mu^v$ and $\nu^v$. If a vertex $v \in F^0$ does not lie on a cycle, let $F_v$ be the empty graph. We claim 
    $$ F'\coloneqq F \cup \Big(\bigcup_{v \in F_0}F_v\Big)$$ 
    satisfies Condition~(K). If $u\in F^0$, Condition~(K) for the graph $F'$ is satisfied at $u$ by construction. For a vertex $u$ on one of the return paths added to $F$, we have some vertex $w \in F^0$ so that $u$ lies on at least one of the return paths $\mu^w$ or $\nu^w$. We consider the following exhaustive cases:
	\begin{enumerate}
		\item $u$ lies on exactly one of $\mu^w$ or $\nu^w$ and is only the range of a single edge in that return path;
		\item $u$ is the range of multiple edges on one of the return paths $\mu^w$ or $\nu^w$;
		\item $u$ lies on both $\mu^w$ and $\nu^w$, and is the range of exactly one edge on each path.
	\end{enumerate}
	We will construct two distinct return paths for $u$ in each case. Let $\mu^w=\mu_1\mu_2\dots\mu_m$ and $\nu^w=\nu_1\nu_2\dots\nu_n$.
	\par Case (i): Without loss of generality, suppose that $u$ lies on $\mu^w$. Let $\mu_i$ be the edge with $r(\mu_i)=u$. Consider the paths $\mu= \mu_i\dots\mu_{m-1}\mu_m\mu_1\dots \mu_{i-1}$ and $\nu=\mu_i\dots\mu_{m-1}\mu_m\nu^w\mu_1\dots \mu_{i-1}$. Note that these are both return paths for $u$ and are distinct as $w$ occurs once on $\mu$ but occurs twice on $\nu$.
	\par Case (ii): Suppose that $u$ is the range of multiple edges on $\mu^w$. Let $\mu_{i_1},\dots,\mu_{i_k}$ be the edges whose range is $u$, with $i_1<i_2<\cdots<i_k$. Let $\mu=\mu_{i_1} \mu_{i_1+1}\dots\mu_{i_2-1}$ and $\nu=\mu_{i_k}\mu_{i_k+1}\dots \mu_m\mu_1\dots \mu_{i_1-1}$. Observe that both $\mu$ and $\nu$ are return paths for $u$, and are distinct as $w$ lies on $\nu$ but does not lie on $\mu$.
	\par Case (iii): Suppose that $\mu_i$ and $\nu_j$ are the edges whose range is $u$. We claim that $\mu=\mu_i \mu_{i+1}\dots \mu_m\mu_1 \dots \mu_{i-1}$ and $\nu=\nu_j \nu_{j+1}\dots \nu_n\nu_1 \dots \nu_{j-1}$ are distinct return paths for $u$. By way of contradiction, suppose that $\mu=\nu$. Then the paths must be the same length, so $m=n$. Furthermore, $w$ occurs exactly once on both $\mu$ and $\nu$. There are $i-1$ edges before $w$ on $\mu$, and $j-1$ edges before $w$ on $\mu$, so $i=j$. Therefore, we have that $\mu_k=\nu_k$ for $1\leq k\leq m$, and so $\mu^w=\nu^w$, which contradicts that they are distinct return paths for $w$. Thus Condition~(K) is satisfied at $u$.
\end{proof}

The next step provides a construction to enlarge a finite subgraph $F \subset E$ to a finite subgraph $\tilde F \subset E$ so that $C^*(\tilde F)$ embeds into $C^*(E)$. By enumerating the edges of $E$, this will produce a direct limit decomposition of $C^*(E)$. This technique is similar to \cite[Definition~1.1 and Lemma~1.2]{RS04}, but as we are in the row-finite, Condition~(K) setting, we provide an alternative and somewhat simpler construction avoiding the use of dual graphs and relying on the Cuntz--Krieger uniqueness theorem (\cite[Theorem~2.13]{CK80}) instead of the gauge-invariant uniqueness theorem (\cite[Theorem~2.3]{aHR97}). The following definition is similar to the ``exit completion'' of Jeong and Park in \cite[Definition~3.2]{JP02}, reformulated for this paper's convention and ensuring Condition~(K) is preserved.

\begin{defi}
	Given a row-finite graph $E$ with Condition~(K) and a finite subgraph $F \subset E$, define a (K)-\emph{entrance completion} of $F$, denoted $\tilde{F}$, as follows. First, add distinct pairs of return paths to all possible vertices as in Lemma~\ref{K_subgraph} to produce a subgraph $F' \subset E$. Then, add to $F'$ all edges $e \in E^1$ so that there is an edge $f$ in $F'$ with $r(f)=r(e)$, along with the source vertices $s(e)$ not already in $F'$. Let $\tilde F \subset E$ be the resulting subgraph and note that $\tilde F$ is finite as $E$ is row-finite.
\end{defi}

The following result is an improvement of Lemma~\ref{K_subgraph}, producing arbitrarily large finite subgraphs of $E$ with Condition~(K) such that the inclusion of the subgraph canonically induces an inclusion of the graph $C^*$-algebras.

\begin{prop}\label{subgraph_iso}
	Given a row-finite graph $E$ with Condition~{\rm (K)} and a finite subgraph $F \subset E$, any ${\rm (K)}$-entrance completion $\tilde{F} \subset E$ of $F$ has Condition~{\rm (K)}. Furthermore, we have that $C^*(\tilde{F})$ is isomorphic to a $C^*$-subalgebra of  $C^*(E)$ containing $\{s_e: e \in F^1\}$ and $\{p_v: v \in F^0\}$.  Explicitly, there is an embedding $C^*(\tilde F) \hookrightarrow C^*(E)$ given by $q_v \mapsto p_v$ and $t_e \mapsto s_e$ for $v \in \tilde F^0$ and $e \in \tilde F^1$, where $(t, q)$ and $(s, e)$ are the universal Cuntz--Krieger $\tilde F$-family and $E$-family, respectively.
\end{prop}

\begin{proof}
	We first show that $\tilde{F}$ has Condition~(K). The subgraph $F' \subset E$ from Lemma~\ref{K_subgraph} contains $F$ and satisfies Condition~(K). Importantly, recall that in the construction of $F'$, every vertex of $F'$ lying on a cycle in $E$ has distinct return paths in $F'$.  Suppose $v$ is a vertex in $\tilde F$ which lies on a cycle in $\tilde F$.  Then $v$ belongs to $F'$; indeed, if $v$ is not in $F'$, then $v$ is a source in $\tilde F$.  By the construction of $F'$, $v$ has at least two return paths in $F'$, and hence also in $\tilde F$, so Condition~(K) holds.

    The projections $\{p_v : v \in \tilde F^0\}$ and partial isometries $\{s_e : e \in \tilde F^1\}$ in $C^*(E)$ form a Cuntz--Krieger $\tilde F$-family.  Indeed,  Definition~\ref{ckdef}\ref{ckdef1} and~\ref{ckdef2} clearly hold, and because any receiver in $\tilde{F}$ receives the same edges in $\tilde{F}$ and $E$, Definition~\ref{ckdef}\ref{ckdef3} is satisfied as well.  Hence there is a $^*$-homomorphism $C^*(\tilde F) \rightarrow C^*(E)$ given on generators by $q_v \mapsto p_v$ and $t_e \mapsto s_e$ for $v \in \tilde F^0$ and $e \in \tilde F^1$.  As each $p_v$ is non-zero and $\tilde F$ has Condition~(K), the Cuntz--Krieger uniqueness theorem (\cite[Theorem~2.13]{CK80}) implies that this $^*$-homomorphism is faithful.
\end{proof}

We now have all of the necessary ingredients to construct the promised inductive limit decompositions of graph $C^*$-algebras associated to graphs with Condition~(K).

\begin{thm}\label{rf_dim2}
 Let $E$ be a countable, row-finite graph with Condition~{\rm (K)}. Then there exists a sequence of finite graphs $E_{i}$ with Condition~{\rm (K)}, giving an inductive limit decomposition $C^*(E) \cong \varinjlim C^*(E_{i})$.
\end{thm}

\begin{proof}
	Let $e_1, e_2, \dots$ be an enumeration of the edges of $E$ and $v_1, v_2, \dots$ be an enumeration of the vertices. Let $F_1$ be a subgraph of $E$ containing $e_1$ and $v_1$. Construct a (K)-entrance completion $\tilde{F_1}$ of $F_1$. Let $F_2$ be a subgraph containing $\tilde{F}_1$, $e_2$, and $v_2$, and construct a (K)-entrance completion $\tilde{F_2}$. Continue in this manner to construct an increasing sequence of finite graphs $\tilde{F_i}$ with Condition~(K). Consider the $C^*$-algebras $C^*(\tilde{F_i})$ as subalgebras of $C^*(E)$ and $C^*(\tilde{F}_{i+1})$ using the $^*$-homomorphisms from Proposition~\ref{subgraph_iso}. As $\{s_e:e \in E^1\}$ and $\{p_v: v \in E^0\}$ are contained in the increasing union $ \bigcup_{i=1}^\infty C^*(\tilde{F}_i)$, we have $C^*(E)=\clo{\bigcup_{i=1}^\infty C^*(\tilde{F}_i)}$. Setting $E_i\coloneqq\tilde{F}_i$ produces the required direct limit decomposition.
\end{proof}

\section{Proof of Theorem \ref{thmA}}\label{sec:thmA}

Having justified a reduction to the case of finite graphs, we now endeavor to show that $\ndim{C^*(E)}\leq 2$ for all directed graphs $E$ with Condition~(K). We begin by showing that for finite graphs with Condition~(K) and no sources, the associated graph $C^*$-algebra is $\cO_\infty$-stable. This will be done by induction using that the $C^*$-algebra associated to a finite graph with Condition~(K) has finitely many ideals.  To facilitate the induction, we need to know the graphs corresponding to ideals and quotients of graph $C^*$-algebras (as in Theorem~\ref{K_ideals}) have Condition~(K) whenever the original graph has Condition~(K).  This is well-known, but we have been unable to find a precise reference in the literature.

\begin{lem}\label{kpasses}
	Let $E$ be a finite graph with Condition~{\rm (K)} and let $H \subset E^0$ be a saturated, hereditary subset. Then the graphs $E_H$ and $E\setminus H$ have Condition~{\rm (K)}.
\end{lem}

\begin{proof}
	Suppose that $v \in E_H^0$ lies on a cycle in $E_H$.  As $v$ also lies in a cycle in $E$, there are distinct return paths $\mu$ and $\nu$ for $v$ in $E$. Because $H$ is hereditary, we have that the vertices of these paths belong to $H$, and so the return paths $\mu$ and $\nu$ in $E$ are actually return paths in $E_H$. Therefore $E_H$ has Condition~(K).
	
    Suppose that $w \in (E\setminus H)^0$ lies on a cycle in $E \setminus H$. As before, $w$ lies on a cycle in $E$, and so $w$ has distinct return paths $\mu$ and $\nu$ in $E$. If any vertex in either return path belongs to $H$, then $w \in H$ as $H$ is hereditary. Thus $\mu, \nu \in (E\setminus H)^*$. We conclude that $E \setminus H$ has Condition~(K).
\end{proof}

The following simple lemma is important for verifying $\cO_\infty$-stability when $E$ is a finite graph with no sources.

\begin{lem}\label{nosources}
	Suppose $E$ is a row-finite graph with no sources and $H \subset E^0$ is a saturated, hereditary subset. Then $E_H$ and $E\setminus H$ have no sources.
\end{lem}

\begin{proof}
	Let $v$ be a vertex in $E_H$. Since $v$ is not a source of $E$, it must receive an edge $e$ from another vertex $u \in E^0$. Because $H$ is hereditary, $u \in H$, and so $e$ and $u$ belong to $E_H$. Therefore, $E_H$ has no sources.
	
    Suppose $w$ is a vertex in $E\setminus H$. As $w$ is not a source of $E$, $r^{-1}(w) \neq \emptyset$.  Let $s(r^{-1}(w))=\{w_1, \dots, w_n\}$. Because $H$ is saturated, there is some $w_j \notin H$. Thus, $w_j$ and the edges from $w_j$ to $w$ belong to $E\setminus H$, and so $E \setminus H$ has no sources.
\end{proof}

We are now ready to show that finite graphs with Condition (K) and no sources produce $\cO_\infty$-stable $C^*$-algebras. Note that by \cite[Theorem A]{BGSW22}, it follows that $\ndim{C^*(E)}=1$ in this case.

\begin{thm}\label{Oinf_stable}
	If $E$ is a finite graph with Condition~{\rm (K)} and no sources, then $C^*(E)$ is $\cO_\infty$-stable. 
\end{thm}

\begin{proof}
	Let $\emptyset=H_0\subset H_1 \subset \cdots \subset H_n =E^0$ be a maximal chain of saturated, hereditary subsets. Note that the graphs $E_{H_i}$ and $E_{H_{i}}\setminus H_{i-1}$ for $1<i\leq n$ have Condition~(K) and no sources by Lemmas~\ref{kpasses} and \ref{nosources}.  Further, each graph \mbox{$E_{H_i}\setminus H_{i-1}$} has no non-trivial saturated, hereditary subsets, and hence the $C^*$-algebra $C^*(E_{H_i} \setminus H_{i-1})$ is simple by Theorem~\ref{K_ideals}.
    Also, because $E_{H_1}$ is a finite graph with no sources, it must contain a cycle. Since $C^*(E_{H_1}) = C^*(E_{H_1} \setminus H_0)$ is simple, \cite[Remark~5.6]{BPRS00} then implies $C^*(E_{H_1})$ is purely infinite. Furthermore, $C^*(E_{H_1})$ is nuclear by \cite[Proposition~2.6]{KP99} (see also \cite[Corollary~4.5.4]{BrownOzawa}). Since $E_{H_1}$ is finite, $C^*(E_{H_1})$ is separable and unital, so it follows that $C^*(E_{H_1})$ is $\cO_\infty$-stable by \cite[Theorem~3.15]{KP00}. By Theorem~\ref{K_ideals}, the corresponding ideal $I_{H_1}$ of $C^*(E_{H_2})$, generated by $\{p_v:v \in H_1\}$, contains a copy of $C^*(E_{H_1})$ as a full corner, and $C^*(E_{H_2}\setminus H_1) \cong C^*(E_{H_2})/I_{H_1}$. By \cite[Theorem~2.8]{B77} and \cite[Corollary~3.2]{TW07}, we have that $I_{H_1}$ is $\cO_\infty$-stable. Using the same argument, $C^*(E_{H_2})/I_{H_1} \cong C^*(E_{H_2}\setminus H_1)$ is $\cO_\infty$-stable. By \cite[Theorem~4.3]{TW07}, $C^*(E_{H_2})$ is $\cO_\infty$-stable, as it is an extension of $\cO_\infty$-stable algebras. Proceeding in this manner proves that $C^*(E_{H_n})=C^*(E)$ is $\cO_\infty$-stable.
\end{proof}

We now prove that a specific choice of $H\subset E^0$ is saturated and hereditary and its complement contains no sources.  Let $E^{\leq \infty}$ denote the set of infinite paths $e_1 e_2 e_3 \dots$ in $E$ together with the set of all finite paths in $E$ whose source is a source in $E$.

\begin{lem}\label{seth}
	Let $E$ be a finite graph with Condition~{\rm (K)}. Let $H \subset E^0$ be the set of vertices $v \in E^0$ such that every path in $E^{\leq \infty}$ with range $v$ is finite.  Then $H$ is saturated and hereditary. Furthermore, $E_H$ contains no cycles, and $E\setminus H$ contains no sources.
\end{lem}	

\begin{proof}
	First, we show that $H$ is hereditary. Suppose $u \in H$ and there is a path $\mu$ from $v \in E^0$ to $u$. Then any path with range $v$ can be extended to a path with range $u$ using $\mu$. Thus there cannot be an infinite path in $E^{\leq\infty}$ with range $v$, so $v \in H$. 
 
    To show that $H$ is saturated, suppose $w \in E^0$ is not a source and satisfies $s(r^{-1}(w)) \subset H$.  Any path $\nu \in E^{\leq \infty}$ with range $w$ must pass through a vertex in $s(r^{-1}(w))$. Removing the first edge from $\nu$ creates a path in $E^{\leq\infty}$ whose range is one of the vertices in $s(r^{-1}(w))$, and by assumption, this path must be finite.  Thus, when $s(r^{-1}(w))\subset H$, it follows that $w \in H$.
	
    By way of contradiction, suppose that $\mu$ is a cycle in $E_H$. Then $ \mu \mu \ldots\in E^{\leq\infty}$ is an infinite path, so $r(\mu)\notin H$. This is a contradiction, so $E_H$ must have no cycles. To show that $E\setminus H$ has no sources, suppose $u$ is a vertex in $E \setminus H$. Because $u \notin H$, we can find an infinite path $\mu\in E^{\leq\infty}$ with range $u$. Each vertex along $\mu$ does not belong to $H$, and so $u$ is not a source in $E\setminus H$.
\end{proof}

We are now ready to prove that all graphs with Condition~(K) have $C^*$-algebras whose nuclear dimension is at most $2$. In the case of a finite graph, Theorem~\ref{Oinf_stable} provides a bound on the nuclear dimension of the quotient, while the ideal is easily verified to have nuclear dimension zero. The Drinen--Tomforde desingularization process (\cite[Definition~2.2]{DT05}) and Theorem~\ref{rf_dim2} are used to prove the result in the case of an infinite graph. 

\begin{proof}[Proof of Theorem~\ref{thmA}]
	We begin with $E$ being a finite graph with Condition~(K).  As in Lemma~\ref{seth}, let $H\subset E^0$ be the set of vertices $v \in E^0$ such that every path in $E^{\leq \infty}$ with range $v$ is finite. Theorem~\ref{K_ideals} implies that $C^*(E_H)$ is isomorphic to a full corner of the ideal $I_H$ generated by $\{p_v:v \in H\}$ and $C^*(E\setminus H) \cong C^*(E)/I_H$.  Furthermore, as $E_H$ is a finite graph with no cycles, $C^*(E_H)$ is a finite dimensional $C^*$-algebra by \cite[Corollary~2.3]{KPR98}, and hence $\ndim{C^*(E_H)}=0$. By \cite[Corollary~2.8]{WZ10}, $$\ndim{I_H}=\ndim{C^*(E_H)}=0.$$  Since $E \setminus H$ has no sources (Lemma~\ref{seth}), $C^*(E\setminus H)$ is $\cO_\infty$-stable by Theorem~\ref{Oinf_stable}. It follows that $\ndim{C^*(E\setminus H)}=1$ by \cite[Theorem~A]{BGSW22}. Therefore $C^*(E)$ is the extension of a $C^*$-algebra with nuclear dimension $1$ by a $C^*$-algebra with nuclear dimension $0$, and so $\ndim{C^*(E)}\leq 2$ by \cite[Proposition~2.9]{WZ10}. 
 
    Now, suppose $E$ is an infinite graph with Condition~(K). Let $F$ be a Drinen--Tomforde desingularization of $E$ as in \cite[Definition~2.2]{DT05}. As $E$ has Condition~(K), \cite[Lemma 2.7]{DT05} implies that $F$ has Condition~(K) and is row-finite. By \cite[Theorem~2.11]{DT05}, $C^*(E)$ is isomorphic to a full corner of $C^*(F)$, and \cite[Corollary~2.8]{WZ10} yields $\ndim{C^*(E)}=\ndim{C^*(F)}$. Using Theorem~\ref{rf_dim2}, construct a sequence of finite graphs $F_i$ with Condition~(K) so that $C^*(F)=\varinjlim C^*(F_i)$. Then, by \cite[Proposition~2.3]{WZ10} and the first half of this proof, we have
	\begin{align*}
		\ndim{C^*(E)}=\ndim{C^*(F)}\leq \liminf \ndim{C^*({F}_i)}\leq 2. &\qedhere
	\end{align*} 
\end{proof}

\section{Proof of Theorem \ref{thmB}}\label{sec:thmB}

For certain extensions, the following result of Evington improves the bound on the nuclear dimension of an extension given in \cite[Proposition~2.9]{WZ10}  by 1.  We will use this to prove Theorem~\ref{thmB} by showing that under certain combinatorial restrictions on the graph, we can drop the upper bound on nuclear dimension from 2 to 1.

\begin{thm}[{\cite[Theorem~1.1]{E22}}]\label{thm:full}
	If $0 \to J \to A \to B \to 0$ is a full extension of a  stable, separable $C^*$-algebra $J$ by a separable, nuclear, $\cO_\infty$-stable $C^*$-algebra $B$, then
	$$1 \leq \ndim{A}\leq \ndim{J}+1.$$
	In particular, if $J$ is a stable AF-algebra, then $\ndim{A}=1$.
\end{thm}

Consider a finite graph $E$ with Condition~(K) and, as in Lemma~\ref{seth}, let $H \subset E^0$ denote the set of vertices $v \in E^0$ such that every path $\mu \in E^{\leq \infty}$ with range $v$ is finite.  Then consider the extension
$$ 0 \longrightarrow I_H \longrightarrow C^*(E) \longrightarrow C^*(E \setminus H) \longrightarrow 0.$$
In the proof of Theorem~\ref{thmA}, we showed $C^*(E\setminus H)$ is $\mathcal O_\infty$-stable and $I_H$ is an AF-algebra.  If $I_H$ is stable and the extension is full, then Theorem~\ref{thm:full} implies $C^*(E)$ has nuclear dimension at most 1.   Our goal is to characterize when  these properties hold in terms of the structure of the graph $E$.  

The following lemma provides a characterization of stability.

\begin{lem}\label{ideal_stable}
    Let $E$ be a finite graph and let $H \subset E^0$ be as in Lemma~\ref{seth}.   Then $I_H$ is stable if and only if every source in $E$ connects to a cycle in $E$. Furthermore, in this case, $I_H \cong \cK^{\oplus m}$, where $m \geq 0$ is the number of sources in $E$.
\end{lem}

\begin{proof}
    By Theorem~\ref{K_ideals}, $I_H = \clo{\mathrm{span}}\{ s_\mu s_\nu^*: s(\mu)=s(\nu) \in H\}$, where $\mu$ and $\nu$ are finite paths in $E$. Given $v \in H$, define $E^*_v$ to be the set of finite paths from a source of $E$ to $v$. Because there is no path from a cycle to $v$, using Definition~\ref{ckdef}\ref{ckdef3} inductively, we may write 
    $$p_v=\displaystyle \sum_{\mu \in E^*_v}s_\mu s_\mu^*.$$ 
    In particular, $I_H = \clo{\text{span}}\{s_\mu s_\nu^*: s(\mu)=s(\nu) \text{ is a source}\}$.  Note that the elements $s_\mu s_\nu^*$ form a system a matrix units and hence give an isomorphism
    $$ I_H \cong \bigoplus_u \mathcal K(l^2(\mu \in E^* : s(\mu) = u)),$$
    where the direct sum is taken over all sources $u \in E^0$. (Recall that $E^*$ denotes the set of all finite paths in $E$.)
    
    If $u \in E^0$ is a source which does not connect to a cycle, then as $E$ is finite, there are only finitely many paths originating from $u$.  Hence $I_H$ has a finite dimensional direct summand, and $I_H$ is not stable.  Conversely, if each source $u$ connects to a cycle, there are infinitely many paths in $E$ originating from $u$, thus $I_H \cong \mathcal K^{\oplus m}$, where $m$ is the number of sources in $E$.  In particular, $I_H$ is stable.
\end{proof}

Having determined when $I_H$ is stable, it remains to characterize fullness of the relevant extension. The following condition provides a complete characterization in the presence of Condition~(K). The condition that $E$ has at least one source rules out the trivial case where $I_H = 0$.

\begin{lem}\label{full_ext}
	Let $E$ be a finite graph with Condition~{\rm (K)} which contains at least one source and satisfies that every source connects to a cycle. Let $H \subset E^0$ be as in Lemma~\ref{seth}. Then the extension $$0 \to I_H \to C^*(E)\to C^*(E\setminus H)\to 0$$ is full if and only every source of $E$ connects to every cycle in $E$.
\end{lem}

\begin{proof}
	We begin by showing the forward direction. By way of contradiction, suppose that the extension is full, but there is a source $u$ and a cycle $\mu$ with no path from $u$ to $\mu$.  Fix a vertex $v$ on $\mu$. By Lemma \ref{ideal_stable}, we may apply \cite[Proposition~3.7]{E22} (which follows from the results of \cite{KN06}) to conclude that the extension is purely large in the sense of \cite[Paragraphs~1 and~2]{EK01}. Furthermore, by \cite[Lemma 7]{EK01}, we have that every positive element in $C^*(E)\setminus I_H$ Cuntz dominates every positive element of $I_H$. Since $v$ lies on a cycle in $E$, $p_v \notin I_H$. Therefore, we have $p_u \preceq p_v$, implying that $p_u \in \clo{C^*(E)p_vC^*(E)}$.  
 
    By Theorem~\ref{K_ideals}, the ideal generated by $p_v$ corresponds to the smallest saturated, hereditary subset $S \subset E^0$ containing $v$.  The set $S$ can be defined explicitly as follows.  Let $S_0 \subset E$ denote the set of all vertices which admit a path to $v$. Note that $u \notin S_0$ and $S_0$ is the smallest hereditary subset of $E$ containing $v$.  By \cite[Remark~4.11]{R05}, the smallest saturated subset of $E^0$ containing $S_0$ is hereditary and hence equals $S$.  Since $u$ is a source and $u \notin S_0$, $E^0 \setminus \{u\}$ is a saturated set containing $S_0$, so $u \notin S$.  We conclude that $p_v \in I_S$ and $p_u \notin I_S$.  Thus $p_u$ is not in the ideal generated by $p_v$, which is a contradiction.
	
	We now show the converse. To this end, we wish to show that $\beta(b)$ is full in $Q(I_H)$ for all non-zero $b \in C^*(E \setminus H)$, where $\beta$ is the Busby map for the extension and $Q(I_H)$ is the corona algebra of $I_H$.  Let $\lambda\colon C^*(E) \to M(I_H)$ denote the canonical extension of the inclusion $I_H \into M(I_H)$. By Lemma~\ref{ideal_stable}, $I_H\cong  \cK^{\oplus m}$, where $m \geq 1$ denotes the number of sources in $E$.  Explicitly, if $u_1, \ldots, u_m$ denote the sources in $E$ and $\mathcal H_i$ denotes the Hilbert space with orthonormal basis $\{\delta_\mu: \mu \in E^*, s(\mu)=u_i\}$, then there is an embedding        $\theta'_i \colon \mathcal K(\mathcal H_i) \rightarrow I_H$
    defined by $\theta'_i(\delta_\mu \delta_\nu^*) = s_\mu s_\nu^*$, where $\delta_\mu \delta_\nu^*$ is the rank one operator $\xi \mapsto \langle \xi, \delta_\nu\rangle \delta_\mu$.  Then the $\theta_i'$ yield an isomorphism 
    \begin{equation*}
        \theta' = \bigoplus_{i=1}^m \theta_i' \colon \bigoplus_{i=1}^m \mathcal K(\mathcal H_i) \rightarrow I_H.
    \end{equation*}
    Let $\theta \colon M(I_H) \rightarrow \mathcal \bigoplus_{i=1}^m \mathcal B(\mathcal H_i)$ be the isomorphism induced by $(\theta')^{-1}$, and let $\theta_i \colon M(I_H) \rightarrow \mathcal B(\mathcal H_i)$ denote the $i$th coordinate of $\theta$ for $ 1 \leq i \leq m$. 

    To show that $\beta(b)$ is full for all nonzero $b \in C^*(E\setminus H)$, it suffices to show that the $\theta_i(\lambda(a))$ is full in $B(\cH_i)$ for all $a \in C^*(E)\setminus I_H$ and $1 \leq i \leq m$.  Given an $ a\in C^*(E) \setminus I_H$, the ideal generated by $a$ is generated by the vertex projections which it contains by Theorem~\ref{K_ideals}. We may then choose a $v \in E^0\setminus H$ so that $p_v \in \clo{C^*(E)aC^*(E)}\setminus I_H$, as otherwise $a$ must belong to $I_H$.  Further, for any $w \in E^0 \setminus H$, since $w$ receives an infinite path in $E^{\leq\infty}$ (by the definition of $H$) and $E$ is finite, there is a vertex $w'$ on a cycle and a path $\mu$ from $w'$ to $w$.  Then $p_{w'} = s_\mu^*s_\mu \sim s_\mu s_\mu^* \leq p_w$, so $p_{w'}$ is in the ideal generated by $p_w$.  So it suffices to show that $\theta_i(\lambda(p_w))$ is full in $B(\mathcal H_i)$ for all $w \in E^0$ lying on a cycle and $1 \leq i \leq m$.

    Note that for each source $u_i$, $\theta_i(\lambda(p_{u_i}))=\theta_i(p_{u_i})$ is the projection of $\cH_i$ onto $\operatorname{span}{(\delta_{u_i})}$, considering $u_i$ as a length zero path. In particular, $\theta_i(\lambda(p_{u_i}))$ is a non-zero projection in $B(\cH_i)$. Fix $w \in E^0$ such that there is a cycle $\mu \in E^*$ based at $w$ and let $1 \leq i \leq m$. By assumption, there is a path $\nu \in E^*$ with $s(\nu) =  u_i$ and $r(\nu) = w$. We may choose $\nu$ so that the cycle $\mu$ is not a subpath of $\nu$, ensuring that $s_\mu$ and $s_\nu$ have orthogonal ranges. Then
    $s_\nu s_\nu^* + s_\mu s_\mu^* \leq p_w$, and hence
    $$\theta_i(\lambda(s_\nu)) \theta_i(\lambda(s_\nu))^* + \theta_i(\lambda(s_\mu)) \theta_i(\lambda(s_\mu))^* \leq \theta_i(\lambda(p_w)).$$
    Since $\theta_i(\lambda(s_\nu))^* \theta_i(\lambda(s_\nu)) = \theta_i(\lambda(p_{u_i})) \neq 0$ we have $$\theta_i(\lambda(p_w))\gneq \theta_i(\lambda(s_\mu)) \theta_i(\lambda(s_\mu))^* \sim \theta_i(\lambda(s_\mu))^* \theta_i(\lambda(s_\mu)) = \theta_i(\lambda(p_w)),$$
    from which it follows that $\theta_i(\lambda(p_w))$ is a projection of infinite rank in $B(\mathcal H_i)$, and hence is full since $\cK(\cH_i)$ is the only non-trivial ideal of $B(\cH_i)$. We conclude that $\theta \circ \lambda$ is full, as required.
\end{proof}

We conclude with the proof of Theorem~\ref{thmB}.  In the setting where each source satisfies Theorem~\ref{thmB}\ref{B1} (i.e., every source connects to every cycle), the bound on the nuclear dimension follows from the work above.  We reduce to this case by showing that the sources which do not connect to any cycles can be split off as a finite dimensional direct summand of the graph $C^*$-algebra.

\begin{proof}[Proof of Theorem~\ref{thmB}]
	If $E$ does not contain a cycle, then by \cite[Corollary~2.3]{KPR98} $C^*(E)$ is finite dimensional and so has nuclear dimension zero. Thus we may assume that $E$ contains a cycle. We wish to reduce to the case where every source connects to every cycle. If this is already the case, we proceed to the next paragraph with $F\coloneq E$. Otherwise, let $S$ denote the set of sources that do not connect to any cycle. Let $H'\subset E^0$ be the set of vertices that only receive paths from $S$ and no other sources. This set is clearly saturated and hereditary, and by Theorem~\ref{K_ideals}, $$I_{H'}=\clo{\text{span}}\{s_\mu s_\nu^*: s(\mu)=s(\nu) \in H'\},$$ where $\mu$ and $\nu$ are finite paths in $E$. However, because no vertex in $H'$ connects to a cycle, there are only finitely many paths $\mu$ with $s(\mu)\in H'$. Applying the argument of \cite[Corollary~2.3]{KPR98}, it follows that $I_{H'}$ is a finite dimensional $C^*$-algebra and, in particular, contains a unit. We then have $C^*(E)\cong I_{H'}\oplus C^*(E\setminus H')$.  Since $I_{H'}$ is finite dimensional, its nuclear dimension is 0, so by \cite[Proposition~2.3]{WZ10}, it suffices to show $C^*(E \setminus H')$ has nuclear dimension 1.
    
    Set $F\coloneqq E \setminus H'$, and note that $F$ has at least one cycle and every source of $F$ connects to every cycle of $F$.  If $F$ has no sources, then $C^*(F)$ is $\mathcal O_\infty$-stable by Theorem~\ref{Oinf_stable}, and hence has nuclear dimension 1 by \cite[Theorem~A]{BGSW22}.  If $F$ has at least one source, define $H \subset F^0$ as in Lemma~\ref{seth} (except for $F$ instead of $E$). Then, by Lemma~\ref{full_ext}, the extension 
    $$0 \longrightarrow I_H \longrightarrow C^*(F)\longrightarrow C^*(F\setminus H)\longrightarrow 0$$
    is full. Additionally, by Lemma~\ref{ideal_stable}, $I_H$ is stable (and separable), and $C^*(F\setminus H)$ is separable, nuclear, and $\cO_\infty$-stable by Theorem~\ref{Oinf_stable}. Therefore, by Theorem~\ref{thm:full} $C^*(F)$ has nuclear dimension 1.
\end{proof}

\providecommand{\bysame}{\leavevmode\hbox to3em{\hrulefill}\thinspace}
\providecommand{\MR}{\relax\ifhmode\unskip\space\fi MR }
\providecommand{\MRhref}[2]{%
	\href{http://www.ams.org/mathscinet-getitem?mr=#1}{#2}
}
\providecommand{\href}[2]{#2}

\end{document}